\documentclass[11pt,reqno]{amsart}

\usepackage[usenames]{color}
\usepackage{amssymb}
\usepackage{graphicx}
\usepackage{amscd}

\usepackage[colorlinks=true,
linkcolor=webgreen,
filecolor=webbrown,
citecolor=webgreen]{hyperref}

\definecolor{webgreen}{rgb}{0,.5,0}
\definecolor{webbrown}{rgb}{.6,0,0}

\usepackage{color}
\usepackage{fullpage}
\usepackage{float}

\usepackage{graphics,amsmath,amssymb}
\usepackage{amsthm}
\usepackage{amsfonts}
\usepackage{latexsym}
\usepackage{epsf}

\begin{document}

\theoremstyle{plain}
\newtheorem{theorem}{Theorem}
\newtheorem{corollary}[theorem]{Corollary}
\newtheorem{lemma}[theorem]{Lemma}
\newtheorem{proposition}[theorem]{Proposition}

\theoremstyle{definition}
\newtheorem{definition}[theorem]{Definition}
\newtheorem{example}[theorem]{Example}
\newtheorem{conjecture}[theorem]{Conjecture}
\newtheorem{notation}[theorem]{Notation}

\theoremstyle{remark}
\newtheorem{remark}[theorem]{Remark}

\title{A FAMILY OF CUBIC DIOPHANTINE EQUATIONS AND 4-CHAINS}


\author{Karen Ge}
\address{}
\curraddr{}
\email{kge@bu.edu}
\thanks{}



\date{}

\dedicatory{}

\begin{abstract}
In a simple integer chain, if $u_{i-1}$, $u_i$, and $u_{i+1}$ are three consecutive terms of the chain, and the pair $(u_{i-1}, u_i)$ has a certain property, then the next pair $(u_i, u_{i+1})$ also has the same property. We extend the idea of a simple chain to an $n$-chain in which $n$ is a positive integer and if a pair $(u_{i-1}, u_{i})$ has a certain property, then the $n$th next pair $(u_{\pm n+i-1}, u_{\pm n+i})$ also has the same property. In this case, we call $(u_{i-1}, u_{i}, u_{i+1})$ and 
$(u_{\pm n+i-1}, u_{\pm n+i}, u_{\pm n+i+1})$ {\it matching triples}. We use $4$-chains to study a family  of cubic Diophantine equations including $x^3 + y^3 + x +y +1 = xyz$ and three others. We show that a pair of integers $(x, y)$ satisfies one of those four equations if and only if $x$ and $y$ are consecutive terms of a $4$-chain. Our main result is that if triple $(u, t, vw)$ is an ordered list of three consecutive terms of one $4$-chain, where $|t|$ is a prime, $t \nmid (u-v)$, and triple $(v, t, uw)$ is that of a second $4$-chain and it matches the first triple, then triple $(-w, t, -uv)$ is that of a third $4$-chain and it matches the other two triples.
\end{abstract}

\maketitle

\section{Introduction}

Many Diophantine equations have been solved in the past 350 years. A comprehensive, though not up-to-date, list of interesting results can be found in Mordell's classical book~\cite{Mordell69}. However, finding general methods of solving nonlinear Diophantine equations is still as challenging as ever. In~\cite{Mills53}, Mills uses integer chains to study the pair of simultaneous quadratic Diophantine relations $x \, | \, y^2 +1$ and $y \, | \, x^2 +1$. By a chain we mean a sequence of integers $(u_n)_{n \in \mathbb{Z}}$ such that if $u_{i-1}$, $u_i$, and $u_{i+1}$ are three consecutive terms of the chain, and the pair $(u_{i-1}, u_i)$ has a certain property, then the next pair $(u_i, u_{i+1})$ also has the same property. Here we call such a sequence a $1$-chain since shifting the indices by $1$ gives another pair of integers that satisfy the same property. Mohanty~\cite{Mohanty77} studies the pair of simultaneous cubic Diophantine relations $x \, | \, y^3 +1$ and $y \, | \, x^3+1$ similarly using $1$-chains. Mills~\cite{Mills53} shows that positive integers $x$ and $y$ satisfy
\[
x \, | \, y^2 +1 \text{ and } y \, | \, x^2 +1
\]
if and only if $x$ and $y$ are consecutive terms of the sequence 1, 1, 2, 5, 13, 34, $\ldots$, obtained from Fibonacci sequence  by striking out alternate terms. In~\cite{Mohanty77}, Mohanty shows that positive integers $x$ and $y$ satisfy
\begin{equation*}
x \, | \, y^3 +1 \text{ and } y \, | \, x^3 +1 
\end{equation*}
if and only if $x$ and $y$ are consecutive terms of an infinite $1$-chain. Dofs~\cite{Dofs93} uses a $1\pm$chain to represent a $1$-chain that has both positive and negative terms and extends Mohanty~\cite{Mohanty77}'s results to integers.

\medskip

We generalize the idea of a simple integer chain and introduce $n$-chains, in which $n$ is a positive integer and if a pair $(u_{i-1}, u_{i})$ has a certain property, then the pairs $(u_{n+i-1}, u_{ n+i})$ and $(u_{-n+i-1}, u_{ -n+i})$  also have the same property. In this case, $(u_{i-1}, u_{i}, u_{i+1})$, 
$(u_{n+i-1}, u_{n+i}, u_{n+i+1})$, and  $(u_{-n+i-1}, u_{-n+i}, u_{-n+i+1})$ are called {\it matching triples}. We use $4$-chains to study cubic Diophantine relations of the type $y \, | \, x^3 + x +1$ and $y \, | \, x^3 +x^2 +1$. We are interested in $x^3 +x +1$ and $x^3+x^2+1$ because unlike $x^3 + 1$, they are irreducible polynomials. The solutions to those Diophantine relations and the relationships among their various solutions are  more elusive.

\medskip

We first introduce  four cubic Diophantine equations and show that each equation is equivalent to a system of simultaneous  cubic Diophantine relations. We define $4$-chains and matching triples formally and show that integers $x$ and $y$ satisfy one of those relations if and only if they are consecutive terms of an infinite $4$-chain. Then we consider $4$-chains that share common elements and  show that there exist two non-identical $4$-chains that have the same non-trivial least element. Finally, we show that if $u$, $t$, $vw$ are three consecutive terms of one $4$-chain, $v$, $t$, $uw$ are three consecutive terms of a second $4$-chain, where $(u, t, vw)$ and $(v, t, uw)$ are matching triples, $|t|$ is a prime, and $t \nmid (u-v)$, then $-w$, $t$, $-uv$ are three consecutive terms of a third $4$-chain. Furthermore, $(u, t, vw)$, $(v, t, uw)$, and $(-w, t, -uv)$ are  matching triples.

\section{Four Systems of Simultaneous Cubic Diophantine Relations}

Now we introduce the four pairs of simultaneous cubic Diophantine relations that we will study in this paper and prove that each pair is equivalent to a cubic Diophantine equation.

\begin{theorem}\label{eq}
The cubic Diophantine equation $x^3 + y^3 + x + y +1 = xyz$ has a solution  if and only if $(x, y)$ satisfies the system $S_{1,1}$:
\[ 
\begin{cases}
\; x \, | \, y^3 + y +1, \\
\; y \, | \, x^3 + x +1.
\end{cases} \tag{$S_{1,1}$}
\]
\end{theorem}

\begin{proof}
If $(x, y, z)$ is a solution to $x^3 + y^3 + x + y +1 = xyz$, then clearly  $x \, | \, y^3 + y +1$ and $y \, | \, x^3 + x +1$. On the other hand, if  $x \, | \, y^3 + y +1$ and $y \, | \, x^3 + x +1$, then $x^3 + y^3 + x + y +1$ is divisible by $x$, divisible by $y$, and $\gcd(|x|, |y|) = 1$. So $x^3 + y^3 + x + y +1$ is divisible by $xy$. Thus there is an integer $z$ such that $x^3 + y^3 + x + y +1 = xyz$.
\end{proof}

Similarly, we can show the following.

\begin{enumerate}

\item The cubic Diophantine equation $x^3 + y^3 + x + y^2 +1 = xyz$ has a solution if and only if  $(x, y)$ satisfies the system $S_{2,1}$: 
\[
\begin{cases}
\; x \, | \, y^3 + y^2 +1,\\
\; y \, | \, x^3 + x +1.
\end{cases} \tag{$S_{2,1}$}
\]

\item The cubic Diophantine equation $x^3 + y^3 + x^2 + y^2 +1 = xyz$ has a solution  if and only if  $(x, y)$ satisfies the system $S_{2, 2}$: 
\[ 
\begin{cases}
\; x \, | \, y^3 + y^2 +1,\\
\; y \, | \, x^3 + x^2 +1.
\end{cases} \tag{$S_{2, 2}$}
\]

\item The cubic Diophantine equation $x^3 + y^3 + x^2 + y +1 = xyz$ has a solution  if and only if $(x, y)$ satisfies the system $S_{1,2}$: 
\[ 
\begin{cases}
\; x \, | \, y^3 + y +1,\\
\; y \, | \, x^3 + x^2 +1.
\end{cases} \tag{$S_{1,2}$}
\]
\end{enumerate}

Clearly, if $(x, y)$ is a solution to $S_{2,1}$, then $(y, x)$ is a solution to $S_{1,2}$. We give them different names  because order matters in the formation  of $4$-chains we will study.  Note that if $(x, y)$ is a solution to any of those four systems of simultaneous Diophantine relations, then $\gcd(|x|, |y|) = 1$. Also note that since neither $x^3 + x + 1 = 0$ nor $x^3 + x^2 + 1 = 0$ has integer solutions, 0 is never a part of a solution to any of those four systems. 

\medskip

Next we show that the solutions of those four systems of cubic Diophantine relations are related.

\begin{theorem}\label{s11}
Let integer pair $(x_0, y_0)$ be a solution to $S_{1,1}$. Let $y_{-1}$ and $x_1$ be real numbers satisfying
\[
y_{-1} y_0= x_0^3 +x_0 +1, \qquad \text{and} \qquad x_0 x_1 = y_0^3 + y_0 +1,  
\]
respectively. Then $(y_{-1}, x_0)$ is a solution to $S_{1,2}$ and $(y_0, x_1)$ is a solution to $S_{2,1}$.
\end{theorem}

\begin{proof}
Since $(x_0, y_0)$ is a solution to $S_{1,1}$, we have $\gcd(|x_0|, |y_0|) = 1$ and 
\[ 
\begin{cases}
\; x_0 \, | \, y_0^3 + y_0 +1,\\
\; y_0 \, | \, x_0^3 + x_0 +1.
\end{cases} 
\]
By the definition of $x_1$, we see that $x_1$ is an integer and $x_1 \, | \, y_0^3 +y_0 + 1$. Furthermore, $x_0x_1 \equiv 1 \pmod{y_0}$. Thus,
\[
x_0^3(x_1^3 + x_1^2 +1) \equiv 1 + x_0 + x_0^3 \equiv 0 \pmod{y_0}.
\]
Since $\gcd(|x_0|, |y_0|) = 1$, we have $y_0 \, | \, x_1^3 + x_1^2 +1$. Thus, $(y_0, x_1)$ is a solution to $S_{2,1}$. 

\medskip

Similarly, by the definition of $y_{-1}$, we see that $y_{-1}$ is an integer and $y_{-1} \, | \, x_0^3 +x_0 + 1$. Since $y_{-1}y_0 \equiv 1 \pmod{x_0}$, we have 
\[
y_{0}^3(y_{-1}^3 + y_{-1}^2+ 1) \equiv 1 + y_{0} + y_{0}^3 \equiv 0 \pmod{x_0}.
\]
Since $\gcd(|x_0|, |y_0|) = 1$, we have $x_{0} \, | \, y_{-1}^3 + y_{-1}^2 +1$. Thus $(y_{-1}, x_0)$ is a solution to $S_{1,2}$.
\end{proof}

The following three corollaries can all be proved using parallel arguments.

\begin{corollary}\label{s21}
Let integer pair $(x_0, y_0)$ be a solution to $S_{2,1}$. Let $y_{-1}$ and $x_1$ be real numbers satisfying
\[
y_{-1}y_0 = x_0^3 +x_0 +1, \qquad \text{and} \qquad x_0 x_1 = y_0^3 + y_0 ^2+1, 
\]
respectively. Then $(y_{-1}, x_0)$ is a solution to $S_{1,1}$ and $(y_0, x_1)$ is a solution to $S_{2,2}$.
\end{corollary}

\begin{corollary}\label{s22}
Let integer pair $(x_0, y_0)$ be a solution to $S_{2, 2}$. Let $y_{-1}$ and $x_1$ be real numbers satisfying
\[
y_{-1}y_0 = x_0^3 +x_0^2 +1, \qquad \text{and} \qquad  x_0 x_1 = y_0^3 + y_0^2 +1, 
\]
respectively. Then $(y_{-1}, x_0)$ is a solution to $S_{2,1}$ and $(y_0, x_1)$ is a solution to $S_{1,2}$.
\end{corollary}

\begin{corollary}\label{s12}
Let integer pair $(x_0, y_0)$ be a solution to $S_{1,2}$. Let $y_{-1}$ and $x_1$ be real numbers satisfying
\[
y_{-1}y_0 =  x_0^3 +x_0^2 +1, \qquad \text{and} \qquad x_0 x_1 = y_0^3 + y_0 +1, 
\]
respectively. Then $(y_{-1}, x_0)$ is a solution to $S_{2, 2}$ and $(y_0, x_1)$ is a solution to $S_{1, 1}$.
\end{corollary}

\section{4-chains}

We are now ready to define $4$-chains formally.

\begin{definition}
An infinite sequence $(u_n)_{n \in \mathbb{Z}}$ is called a $4$-chain if and only if
\[
u_{n-1} u_{n+1} = u_{n}^3 + u_{n}^{f(n)} + 1, 
\]
where 
\[
f(n) = \begin{cases}
        1, & \text{ if $n \equiv 0,3 \pmod{4}$}; \\
        2, & \text{ if $n \equiv 1,2 \pmod{4}$}.
			 \end{cases}
\] Two $4$-chains $(u_n)_{n \in \mathbb{Z}}$, $(v_n)_{n \in \mathbb{Z}}$ are considered the same if and only if there exists a $k$ such that either $u_n = v_{k+n}$ for all $n$ or $u_{n} = v_{k-n}$ for all $n$.
\end{definition}

As an example, let's use the pair $(x_0, y_0) = (-1, -1)$, a solution to $S_{2, 2}$, to build a $4$-chain. Using the notation of Corollary~\ref{s22}, we have
\[
y_{-1} = \frac{x_0^3 +x_0^2 +1}{y_0} = -1, \qquad \text{and} \qquad x_1 = \frac{y_0^3 + y_0^2 +1}{x_0} = -1 .
\]
Thus, we get 

\begin{table}[h]
\centering
\begin{tabular}{cccccc}
$y_{-1}$ & $x_0$ & $y_0$ & $x_1$ \\
$-1$ & $-1$ & $-1$ & $-1$ 
\end{tabular}
\end{table}
 Here $(y_{-1}, x_0)$, $(x_0, y_0)$, and $(y_0, x_1)$ are solutions to $S_{2, 1}$, $S_{2, 2}$, and $S_{1, 2}$, respectively.

Next, we apply Corollary~\ref{s21} to $(y_{-1}, x_0)$, and Corollary~\ref{s12} to $(y_0, x_1)$ to find
\[
x_{-1} = \frac{y_{-1}^3 + y_{-1} + 1}{x_0} = 1, \qquad \text{and} \qquad y_1 = \frac{x_1^3 + x_1 + 1}{y_0} = 1 .
\]

Now we have

\begin{table}[H]
\centering
\begin{tabular}{cccccc}
$x_{-1}$ & $y_{-1}$ & $x_0$ & $y_0$ & $x_1$ & $y_1$ \\
1 & $-1$ & $-1$ & $-1$ & $-1$ & 1
\end{tabular}
\end{table}

Repeating the process, we get an infinite chain

\begin{table}[h]
\centering
\begin{tabular}{cccccccccccccc}
$\cdots$ & $y_{-3}$ & $x_{-2}$ & $y_{-2}$ & $x_{-1}$ & $y_{-1}$ & $x_0$ & $y_0$ & $x_1$ & $y_1$ & $x_2$ & $y_2$ & $x_3$  & $\cdots$\\
$\cdots$ & 1541 & $-17$ & $-3$ & 1 & $-1$ & $-1$ & $-1$ & $-1$ &1 & $-3$ & $-17$ & 1541 &$\cdots$
\end{tabular}
\end{table}

In the list above, all pairs $(x_{2i}, y_{2i})$ are solutions to $S_{2, 2}$, all pairs $(y_{2i}, x_{2i+1})$ are solutions to $S_{1, 2}$, and so on. Thus the sequence above is a $4$-chain. Therefore, each of the four systems of simultaneous cubic Diophantine relations has infinitely many integer solutions. 

\medskip

As another example, if we use the pair $(x_0, y_0) = (-1, -1)$, a solution to $S_{1, 1}$, to build a $4$-chain, we get

\begin{table}[H]
\centering
\begin{tabular}{cccccccccccc}
$\cdots$ & $x_{-2}$ & $y_{-2}$ & $x_{-1}$ & $y_{-1}$ & $x_0$ & $y_0$ & $x_1$ & $y_1$ & $x_2$ & $y_2$  & $\cdots$\\
$\cdots$ & 1643 & $-17$ & $-3$ & 1 & $-1$  & $-1$  & 1 & $-3$ & $-17$ & 1643 &$\cdots$
\end{tabular}
\end{table}

In this  $4$-chain, all pairs $(x_{2i}, y_{2i})$ are solutions to $S_{1, 1}$, all pairs $(y_{2i}, x_{2i+1})$ are solutions to $S_{2, 1}$, and so on.

\medskip

We see that two integers $x$ and $y$  satisfy one of the four systems ($S_{1, 1}$, $S_{2, 1}$, $S_{2, 2}$, or $S_{1, 2}$) of cubic Diophantine relations if and only if they are consecutive terms of a $4$-chain. Moreover, any two consecutive terms of a $4$-chain together with the system  they satisfy determine the $4$-chain completely. Therefore, the problem of finding the solutions of those four systems of cubic Diophantine relations is the same as the problem of determining all $4$-chains.

\begin{notation}
We use $\langle u_{i}, u_{i +1}\rangle_{S_{\lambda_a, \lambda_b}}$ to represent the $4$-chain generated by the pair $(u_i, u_{i+1})$ satisfying the system $S_{\lambda_a, \lambda_b}$, where $\lambda_a$ and $\lambda_b$ are constants taking values in $\{1, 2\}$.
\end{notation}

An argument similar to that in the proof of Theorem~\ref{s11} gives us the following.

\begin{corollary}\label{chain}
Let $u$, $v$, $w$, in that order, be three consecutive terms of a $4$-chain. Then the pair $(u, v)$ satisfies the system $S_{\lambda_a, \lambda_b}$ if and only if the pair $(v, w)$ satisfies the system $S_{3-\lambda_b, \lambda_a}$. In other words, 

\[
\langle u, v\rangle_{S_{\lambda_a, \lambda_b}} = \langle v, w\rangle_{S_{3-\lambda_b, \lambda_a}} .
\]
\end{corollary}

When we examine consecutive pairs of an infinite $4$-chain pair by pair, we see that the systems of cubic Diophantine relations they satisfy are the following:

\[
\cdots \leftrightarrow S_{\lambda_a,\lambda_b} \leftrightarrow S_{3-\lambda_b,\lambda_a} \leftrightarrow S_{3-\lambda_a,3-\lambda_b} \leftrightarrow S_{\lambda_b,3-\lambda_a} \leftrightarrow S_{\lambda_a,\lambda_b} \leftrightarrow \cdots
\]

Indeed, any $4$-chain is a reversible $4$-chain. 

\medskip

Our last definition is that of  matching triples.

\begin{definition}
If $u$, $v$, $w$, in that order, are three consecutive terms of  $\langle u, v\rangle_{S_{\lambda_a,\lambda_b}}$, and $x$, $y$, $z$, in that order, are three consecutive terms of $\langle x, y\rangle_{S_{\lambda_a,\lambda_b}}$, which is not necessarily the same chain as $\langle u, v\rangle_{S_{\lambda_a,\lambda_b}}$, then we call $(u, v, w)$ and $(x, y, z)$ {\it matching triples}. 
\end{definition}

\section{Basic Properties of 4-chains}

It's easy to check that the only $4$-chain that has three or more consecutive terms that are the same is
\[
\ldots, \; 1541, \; -17, \; -3, \; 1, \; -1, \; -1, \; -1, \; -1, \; 1, \; -3, \; -17, \; 1541, \; \ldots
\]

When we re-write a portion of the $4$-chain above modulo 3, we get
\[
1 \quad 2 \quad 2 \quad 2 \quad 2 \quad 1
\]
Simple modular arithmetic calculation tells us that this is the longest sub-chain that does not contain zero (mod 3) that any $4$-chain could possibly have.
Also we notice that for any integer $x$, $x^3$ and $x^2+1$ have opposite parities and $x^3$ and $x+1$ have opposite parities. Thus $x^3 + x + 1$ and $x^3 + x^2 +1$ are never even for any integer $x$. Similarly, we can check and see that for any integer $x$, $x^3 + x + 1$ and $x^3 + x^2 + 1$ are never zero modulo 5 or modulo 7. Thus we have proved the following.

\begin{theorem} In any $4$-chain, there are infinitely many terms that are divisible by $3$; there are no terms that are divisible by $2$, or by $5$, or by $7$.
\end{theorem}

Using basic algebra, we can prove the following properties of $4$-chains.

\medskip

In a $4$-chain $(u_i)$, if there exist exactly two consecutive terms that are the same, i.e., if $x \neq y$, and $x$, $x$, $y$ are three consecutive terms of $(u_i)$, then we get 
\[
xy = x^3 + x + 1, \; \; \text{ or }  \; \; xy = x^3 + x^2 + 1 .
\]

 The solutions to $xy = x^3 + x + 1$,  $x \neq y$ are $(x, y) = (1, 3)$ and $(x, y) = (-1, 1)$. The solution to $xy = x^3 + x^2 + 1$,  $x \neq y$ is $(x, y) = (1, 3)$.

\medskip

When $(x, y) = (1, 3)$, we get three $4$-chains with exactly two consecutive terms that are the same.
\begin{enumerate}
\item[(1)] $\ldots, 10251, 31, 3, 1, 1, 3, 37, 16897, \ldots$
\item[(2)] $\ldots, 9941, 31, 3, 1, 1, 3, 31, 9941, \ldots$
\item[(3)] $\ldots, 17341, 37, 3, 1, 1, 3, 37, 17341, \ldots$
\end{enumerate}

When $(x, y) = (-1, 1)$, we get another $4$-chain with exactly two  consecutive terms that are the same.
\begin{enumerate}
\item[(4)] $\ldots, 1643, -17, -3, 1, -1, -1, 1, -3, -17, 1643, \ldots$
\end{enumerate}

If in a $4$-chain $(u_i)$, there is an $n$ such that $u_{n-1} = u_{n+1}$, $u_{n-1} \neq u_n$,  i.e., if  $x \neq y$, and $y$, $x$, $y$ are three consecutive terms of $(u_i)$, then we have
\[
y^2 = x^3 + x + 1, \; \; \text{ or }  \; \; y^2 = x^3 + x^2 + 1 .
\]

Those are the equations of elliptic curves. The solutions to $y^2 = x^3 + x + 1$  are $(x, y) = (0, 1)$ and $(x, y) = (72, 611)$. Neither solution generates a $4$-chain. The solutions to $y^2 = x^3 + x^2 + 1$ are $(x, y) = (-1, 1)$, $(x, y) = (0, 1)$, and $(x, y) = (4, 9)$. Among the three, only $(x, y) = (-1, 1)$ generates a $4$-chain. It is
\[
 \ldots, \; 7849, \; -29, \; -3, \; 1, \; -1, \; 1, \; -3, \; -29, \; 8139, \; \ldots
\]

\section{The Least Elements of 4-chains}

Following the convention of Dofs~\cite{Dofs93}, we define the least element of a $4$-chain $(u_i)$ as the single element with the least absolute value. This is well-defined when the least absolute value of an element is not 1. We note that in a $4$-chain $(u_i)$, if $|u_i| = |u_{i+1}|$, for some $i$, then $u_i = 1$ or $u_i = -1$. In such $4$-chains, several elements have absolute value 1. We call such least elements trivial. We are more interested in $4$-chains with non-trivial least elements.

Note that in a $4$-chain with a non-trivial least element, if $1 <|u_i| < |u_{i+1}|$, then from $u_i \cdot u_{i+2} = u_{i+1}^3 + u_{i+1}^{\lambda_a} + 1$ we get
\[
|u_{i+2}| = \frac{|u_{i+1}^3 + u_{i+1}^{\lambda_a} + 1|}{|u_i|} \geq \frac{|u_{i+1}|^3 - |u_{i+1}|^2 -1}{|u_i|} > |u_{i+1}| ,
\]
where the last inequality comes from the fact that $|u_{i+1}| \geq |u_i| + 1$.

\medskip

Similarly, if $|u_i| > |u_{i+1}| > 1$, then  $u_{i-1} \cdot u_{i+1} = u_{i}^3 + u_i^{\lambda_b}+ 1$ gives us $|u_{i-1}| > |u_{i}|$. Thus such $4$-chains have single least elements. We call the least element of a $4$-chain $u_0$. So for $|u_0| \neq 1$, the $4$-chain
\[
\cdots \qquad u_{- 3} \qquad u_{- 2} \qquad u_{-1} \qquad u_0 \qquad  u_1 \qquad u_2 \qquad u_3 \qquad \cdots
\]
has the property 
\[
\cdots > |u_{- 3}| > |u_{- 2}| > |u_{-1}| > |u_0| \; \; \text{ and } \; \; |u_0| <  |u_1| < |u_2| < |u_3| < \cdots
\]

At the end of his paper~\cite{Mohanty77}, Mohanty asks if there are two non-identical 1-chains with the same least element. Dofs~\cite{Dofs93} gives an affirmative answer to this question. Before we answer the same question for $4$-chains, let's first prove the following simple but useful result.

\begin{lemma}\label{common}
If $u_1, t, u_2$, in that order, are three consecutive terms of $\langle u_1, t\rangle_{S_{\lambda_a,\lambda_b}}$ and  $v$ is a factor of $u_1u_2$, then there is a $4$-chain $\langle v, t\rangle_{S_{\lambda_a, \lambda_b}}$ if and only if $t \, | \, v^3 + v^{\lambda_b} + 1$.
\end{lemma}

\begin{proof} If there is a $4$-chain $\langle v, t\rangle_{S_{\lambda_a, \lambda_b}}$, then clearly $t \, | \, v^3 + v^{\lambda_b} + 1$. Conversely,
note that $u_1u_2 = t^3 + t^{\lambda_a} + 1$. Since $v$ is a factor of $u_1u_2$, we have $v \, | \, t^3 + t^{\lambda_a} + 1$. This together with $t \, | \, v^3 + v^{\lambda_b} + 1$ gives us the $4$-chain $\langle v, t\rangle_{S_{\lambda_a, \lambda_b}}$.
\end{proof}

Note that in Lemma~\ref{common}, $t \, | \, u_1^3 + u_1^{\lambda_b} + 1$ . So $t \, | \, v^3 + v^{\lambda_b} + 1$ is equivalent to 
\[
t \, | \, (v^3 + v^{\lambda_b} + 1) - (u_1^3 + u_1^{\lambda_b} + 1). 
\]

When $\lambda_b = 1$, $(v^3 + v^{\lambda_b} + 1) - (u_1^3 + u_1^{\lambda_b} + 1) = (v-u_1)(v^2 + vu_1 + u_1^2 + 1)$. When $\lambda_b = 2$, $(v^3 + v^{\lambda_b} + 1) - (u_1^3 + u_1^{\lambda_b} + 1) = (v-u_1)(v^2 + vu_1 + u_1^2 + v + u_1)$. 
Thus we have the following corollary. It provides  a way to construct two distinct $4$-chains that share a common element.

\begin{corollary}\label{equiv}
If $u_1, t, u_2$, in that order, are three consecutive terms of $\langle u_1, t\rangle_{S_{\lambda_a,\lambda_b}}$,  $v$ is a factor of $u_1u_2$ and $ t \, | \, (u_1 -v)$, then there is a $4$-chain $\langle v, t\rangle_{S_{\lambda_a, \lambda_b}}$.
\end{corollary}

As an example, the $4$-chain  $\langle -31, -11\rangle_{S_{2, 1}}$ has $-31$, $-11$, 39 as its three consecutive terms. 13 is a factor of $(- 31) \cdot 39$ and $ (-11) \, | \, (-31-13)$. Thus $\langle 13, -11\rangle_{S_{2, 1}}$ is a second $4$-chain with three consecutive terms 13, $-11$, $-93$. That is:

\begin{table}[H]
\centering
\setlength{\tabcolsep}{12pt}
\begin{tabular}{rrrrr}
$\ldots$ & $-31$ & $-11$ & 39 & $\ldots$\\
$\ldots$ & 13 & $-11$ & $-93$ & $\ldots$
\end{tabular}
\end{table}

 Note that $(-31, -11, 39)$ and $(13, -11, -93)$ are matching triples and $-11$ is the least element of both $4$-chains. Thus we have proved the following.

\begin{corollary}
There are non-identical $4$-chains that have the same non-trivial least element.
\end{corollary}

\section{4-chains with Common Elements}

In this section, we prove the main result of this paper.

\begin{theorem}\label{three} If $u, t, vw$, in that order, are three consecutive terms of the $4$-chain $\langle u, t\rangle_{S_{\lambda_a,\lambda_b}} = $ $\langle t, vw\rangle_{S_{3-\lambda_b,\lambda_a}}$, then $v, t, uw$, in that order, are three consecutive terms of the $4$-chain $\langle v, t\rangle_{S_{\lambda_a,\lambda_b}} = $ $\langle t, uw\rangle_{S_{3-\lambda_b,\lambda_a}}$ if and only if
\begin{align*}
&\; t \, | \, (u-v)(u^2 + uv + v^2 + 1), &\qquad \text{ when } \; \lambda_b = 1,\\
&\; t \, | \, (u-v) \big((uw)^2 + (uw)(vw) + (vw)^2 + 1\big), &\qquad \text{ when } \; \lambda_b = 2.
\end{align*}
\end{theorem}

\begin{proof} We first prove the case for $\lambda_a = \lambda$ and $\lambda_b = 1$. If $\langle u, t\rangle_{S_{\lambda,1}}$ and $\langle v, t\rangle_{S_{\lambda,1}}$ are two $4$-chains, then
\[
\begin{cases}
\; t \, | \, u^3 + u +1, \\
\; t \, | \, v^3 + v +1.
\end{cases} 
\]
Subtracting, we get 
\[
t \, | \, u^3 -v^3 + u -v, \qquad \text{ or } \qquad t \, | \, (u-v)(u^2 + uv + v^2 + 1) .
\]

Conversely, if $t \, | \, (u-v)(u^2 + uv + v^2 + 1)$, then
\[
t \, | \, u^3 -v^3 + u -v, \qquad \text{ or } \qquad t \, | \, (u^3 +u +1)-(v^3 + v+1).
\]
Since $\langle u, t\rangle_{S_{\lambda,1}}$ is a $4$-chain, $t \, | \, u^3 + u +1$. So $t \, | \, v^3 + v +1$. Since it is also true that $v \, | \, uvw = t^3 + t^{\lambda} + 1$, we get a $4$-chain  $\langle v, t\rangle_{S_{\lambda,1}}$.

\medskip

For the case $\lambda_a = \lambda$ and $\lambda_b = 2$, we can re-write the $4$-chains as $\langle vw, t\rangle_{S_{\lambda, 1}} = \langle t, u\rangle_{S_{2,\lambda}}$ and 
$\langle uw, t\rangle_{S_{\lambda, 1}} = \langle t, v\rangle_{S_{2,\lambda}}$, respectively. The rest of the proof mirrors that of the first case and the observation that $\gcd(|t|, |w|) = 1$.
\end{proof}

We are now ready to construct a third $4$-chain when given two $4$-chains $\cdots, \, u, \, t, \, vw, \, \cdots$ and $\cdots, \, v, \, t, \, uw, \, \cdots$, where $(u, t, vw)$ and $(v, t, uw)$ are matching triples.

\begin{theorem}\label{threeOne} If $u, t, vw$, in that order, are three consecutive terms of $\langle u, t\rangle_{S_{\lambda,1}}$ $= \langle t, vw\rangle_{S_{2,\lambda}}$, and $v, t, uw$, in that order, are three consecutive terms of a second $4$-chain  $\langle v, t\rangle_{S_{\lambda,1}}$ $= \langle t, uw\rangle_{S_{2,\lambda}}$,  $|t|$ is a prime, $t \nmid (u-v)$, then there is third $4$-chain, namely,
\[
\langle -w, t\rangle_{S_{\lambda,1}} = \langle t, -uv\rangle_{S_{2,\lambda}}.
\]
\end{theorem}

\begin{proof} We may assume that $t \nmid \big(u - (-w)\big)$ and $t \nmid \big(v-(-w)\big)$ because otherwise the theorem is true by Corollary~\ref{equiv}. 
Let $p = |t|$. From $\langle u, t\rangle_{S_{\lambda,1}} = \langle t, vw\rangle_{S_{2,\lambda}}$, we get
\begin{align}
u^3 + u + 1 &\equiv 0 \pmod{p}, \label{a}\\
(vw)^3 + (vw)^2 + 1 &\equiv 0 \pmod{p}. \label{b}
\end{align}
From $\langle v, t\rangle_{S_{\lambda,1}} = \langle t, uw\rangle_{S_{2,\lambda}}$, we get
\begin{align}
v^3 + v + 1 &\equiv 0 \pmod{p}, \label{c}\\
(uw)^3 + (uw)^2 + 1 &\equiv 0 \pmod{p}. \label{d}
\end{align}
Equation~\eqref{a} minus Equation~\eqref{c} gives us 
\begin{equation}
u^3 - v^3 + u -v \equiv 0 \pmod{p}. \label{e}
\end{equation}
Equation~\eqref{d} minus Equation~\eqref{b} gives us $(uw)^3 - (vw)^3 + (uw)^2 -(vw)^2 \equiv 0 \pmod{p}$, or 
\begin{align}
w(u^3-v^3) + u^2 -v^2 \equiv 0 \pmod{p}. \label{f}
\end{align}
Equation~\eqref{f} minus (Equation~\eqref{e} multiplied by $w$) gives us
\begin{align*}
u^2 - v^2 -w(u-v) &\equiv 0 \pmod{p}, \text{ or}\\
(u-v)(u+v-w) &\equiv 0 \pmod{p}.
\end{align*}
Since $t \nmid u-v$, we have $w = u + v \pmod{p}$.

\medskip

By the first part of Theorem~\ref{three}, our theorem is true if 
\[
t \, | \, \big(u-(-w)\big)\big(u^2 + u(-w) + (-w)^2 +1\big) .
\]
Since $|t| = p$ is a prime number and $t \nmid \big(u - (-w)\big)$, we only need to show that 
\begin{equation}
p \, | \, u^2 + u(-w) + (-w)^2 +1 \label{A}
\end{equation}
Since $w = u + v \pmod{p}$, we have 
\begin{align*}
u^2 + u(-w) + (-w)^2 +1 &\equiv u^2 + u\big(-(u+v)\big) + (u+v)^2 + 1 \pmod{p}\\
	&\equiv u^2 - u^2 - uv + u^2 + 2uv + v^2 + 1 \pmod{p} \\
	&\equiv u^2 + uv + v^2 + 1 \pmod{p}.
	\end{align*}
Note that $t \nmid (u-v)$. By Theorem~\ref{three}, $u^2 + uv + v^2 + 1 \equiv 0 \pmod{p}$. Thus \eqref{A} is true and we are done.
\end{proof}

\begin{theorem}\label{threeTwo} If $uw$, $t$, $v$, in that order, are three consecutive terms of  $\langle uw, t\rangle_{S_{\lambda,1}} = \langle t, v\rangle_{S_{2,\lambda}}$, and $vw$, $t$, $u$, in that order, are three consecutive terms of a second $4$-chain $\langle vw, t\rangle_{S_{\lambda,1}} = \langle t, u\rangle_{S_{2,\lambda}}$,  $|t|$ is a prime, $t \nmid (u-v)$, then there is third $4$-chain, namely,
\[
\langle -uv, t\rangle_{S_{\lambda,1}} = \langle t, -w\rangle_{S_{2,\lambda}}.
\]
\end{theorem}

\begin{proof}
As in the proof of the previous theorem, we may assume that $t \nmid \big(uw - (-uv)\big)$ and $t \nmid \big(vw - (-uv)\big)$ because otherwise the theorem is true by Corollary~\ref{equiv}. Let $|t| = p$.  Similar to the proof of the previous theorem, from
\[
\langle uw, t\rangle_{S_{\lambda,1}} = \langle t, v\rangle_{S_{2,\lambda}} \quad \text{and} \quad \langle vw, t\rangle_{S_{\lambda,1}} = \langle t, u\rangle_{S_{2,\lambda}},
\]
we get
\begin{align*}
(uw)^3 - (vw)^3 + uw - vw &\equiv 0 \pmod{p}, \\
u^3 - v^3 + u^2 - v^2 &\equiv 0 \pmod{p}.
\end{align*}
Note that $p$ is prime, $\gcd(p, |w|) = 1$, and $p \nmid (u-v)$. In the two equations above, we use the second equation multiplied by $w^3$ minus the first equation to get
\begin{align*}
w^3(u^2 -v^2) - (uw-vw)&\equiv 0 \pmod{p}, \\
w^2(u^2 -v^2) - (u-v) &\equiv 0 \pmod{p}, \\
w^2(u+v) -1 &\equiv 0 \pmod{p}.
\end{align*}
Since $uvw = t^3 + t^\lambda + 1$, $uvw \equiv 1 \pmod{p}$. So the last equation gives us
\begin{align*}
w^2(u+v)  &\equiv 1 \pmod{p}, \\
uv\big(w^2(u+v)\big)  &\equiv uv \pmod{p}, \\
uw + vw &\equiv uv \pmod{p}.
\end{align*}
Since $p \nmid \big(uw - (-uv)\big)$, by the second part of Theorem~\ref{three}, our theorem is true if 
\begin{equation}
p \, | \, (uw)^2 + (uw)(-uv) + (-uv)^2 +1. \label{B}
\end{equation}
Since $uw + vw \equiv uv \pmod{p}$, we have 
\begin{align*}
&(uw)^2 + (uw)(-uv) + (-uv)^2 +1 \\
&\equiv (uw)^2 + (uw)(-uw-vw) + (uw+vw)^2 + 1 \pmod{p} \\
&\equiv (uw)^2 + (uw)(vw) + (vw)^2 + 1 \pmod{p}.
\end{align*}
By Theorem~\ref{three}, $(uw)^2 + (uw)(vw) + (vw)^2 + 1 \equiv 0 \pmod{p}$. Thus \eqref{B} is true and we are done.
\end{proof}

Combining Theorem~\ref{threeOne} and Theorem~\ref{threeTwo}, we get the following theorem.

\begin{theorem}[Main result]
If $u$, $t$, $vw$, in that order, are three consecutive terms of $\langle u, t\rangle_{S_{\lambda_a,\lambda_b}}$ $= \langle t, vw\rangle_{S_{3-\lambda_b,\lambda_a}}$, and $v$, $t$, $uw$, in that order, are three consecutive terms of a second $4$-chain $\langle v, t\rangle_{S_{\lambda_a, \lambda_b}} = \langle t, uw\rangle_{S_{3-\lambda_b,\lambda_a}}$,  $|t|$ is a prime, $t \nmid (u-v)$, then there is third $4$-chain, namely,
\[
\langle -w, t\rangle_{S_{\lambda_a,\lambda_b}} = \langle t, -uv\rangle_{S_{3-\lambda_b,\lambda_a}}.
\]
\end{theorem}

\medskip

For instance, the three $4$-chains below share the common element 31. Moreover, $(-17, \, 31, \, -3^3 \cdot 67)$, $(3,\;  31, \; 3^2\cdot 17\cdot 67)$ and $(3^2\cdot 67, \; 31, \; 3\cdot 17)$ are three matching triples.
\begin{table}[H]
\centering
\setlength{\tabcolsep}{18pt}
\begin{tabular}{rrrrr}
$\ldots$ & $-17$ & 31 & $-3^3 \cdot 67$ & $\ldots$\\
$\ldots$ & 3 & 31 & $3^2\cdot 17\cdot 67$ & $\ldots$\\
$\ldots$ & $3^2\cdot 67$ & 31 & $3\cdot 17$ & $\ldots$
\end{tabular}
\end{table}

\section*{Acknowledgments}

I am extremely grateful to Professor Izzet Coskun for his immense knowledge and continued guidance. Professor Keith Conrad's comments on an earlier version of the manuscript greatly improved the content and presentation of the paper, although any errors are my own.

\end{document}